\documentclass[12pt]{article}
\usepackage{amsmath, amsthm}
\usepackage{enumerate}
\usepackage{amssymb}
\usepackage{geometry}
\newtheorem{thm}{Theorem}[section]
\newtheorem{exam}[thm]{Example}

\newtheorem{lem}[thm]{Lemma}
\newtheorem{prop}[thm]{Proposition}

\newtheorem{defn}[thm]{Definition}
\newtheorem{rem}[thm]{Remark}

\newcommand{\eval}[2][\right]{\relax
 \ifx#1\right\relax \left.\fi#2#1\rvert}

\begin{document}

\title{\bf Topological algebras of statistical $\tau$-bounded operators on ordered topological vector spaces} 
\maketitle

\author{\centering Abdullah Ayd\i n$^*$ and Muhammed \c{C}{\i}nar\\ \bigskip \small 
	Department of Mathematics, Mu\c{s} Alparslan University, Mu\c{s}, Turkey \\}

\bigskip

\abstract{In this paper, we introduce statistical bounded set on topological vector space. Also, we consider three classes of bounded operators from topological vector spaces to ordered topological vector spaces. Moreover, we give relations between them and order bounded operators. We give algebraic properties of these operators concerning to the uniform convergence topology.}

\bigskip
\let\thefootnote\relax\footnotetext
{ 
	
\text{2010 AMS Mathematics Subject Classification:} 47B65, 46A40, 46H35

$^*$Corresponding author: a.aydin@alparslan.edu.tr}

\section{Introductory Facts}
Bounded operators and the statistical convergence are natural and efficient tools in the theory of functional analysis. A vector lattice that was introduced by F. Riesz \cite{Riez} is an ordered vector space that has many applications in measure theory, operator theory, and applications in economics; see for example \cite{AB,ABPO,AAydn2,Riez,Za}. On the other hand, the statistical convergence is a generalization of the ordinary convergence of a real sequence; see for example \cite{AAydn3,AP,MK}. Studies related to this paper are done by Troitsky \cite{Tr} where the $bb$- and $nb$-bounded operators were defined between topological vector spaces and by Ayd\i n \cite{Aybdd} in which the $ob$-bounded operator was defined from vector lattices to locally solid vector lattice and by Hejazian et al. \cite{HMZ}, where $nb$- and $bb$-bounded operators were defined on topological vector spaces. Also, Albayrak and Pehlivan introduced the statistical $\tau$ bounded on locally solid vector lattice in \cite{AP} and Ayd\i n defined $st$-$u_\tau$-closed set in \cite{AAydn3}. Our aim in this paper is to introduce and study some classes of bounded operators from topological vector spaces to topological ordered vector spaces.

Now, let give some basic notations and terminologies that will be used in this paper. A neighborhood of an element $x$ in a topological vector space $E$ is a subset of $E$ containing an open set that contains $x$. Neighborhoods of zero will often be referred to as zero neighborhoods. Every linear topology $\tau$ on a vector space $E$ has a base $\mathcal{N}$ of zero neighborhoods satisfying the following four properties; for each $V\in \mathcal{N}$, we have $\lambda V\subseteq V$ for all scalar $\lvert \lambda\rvert\leq 1$; for any $V_1,V_2\in \mathcal{N}$ there is another $V\in \mathcal{N}$ such that $V\subseteq V_1\cap V_2$; for each $V\in \mathcal{N}$ there exists another $U\in \mathcal{N}$ with $U+U\subseteq V$; for any scalar $\lambda$ and each $V\in \mathcal{N}$, the set $\lambda V$ is also in $\mathcal{N}$; for much more detail see \cite{AB,ABPO,Tr}. In this article, unless otherwise, when we mention a zero neighborhood, it means that it always belongs to a base that holds the above properties. Let $E$ be a real-valued vector space. If there is an order relation "$\leq$" on $E$, i.e., it is antisymmetric, reflexive and transitive then $E$ is called ordered vector space whenever the following conditions hold: for every $x,y\in E$ such that $x\leq y$, we have $x+z\leq y+z$ and $\alpha x\leq\alpha y$ for all $z\in E$ and $\alpha \in \mathbb{R}$. An ordered vector space $E$ is called Riesz space or vector lattice if, for any two vectors $x,y\in E$, the infimum $x\wedge y=\inf\{x,y\}$ and the supremum $x\vee y=\sup\{x,y\}$ exist in $E$. Let $E$ be a vector lattice. Then, for any $x\in E$, the positive part of $x$ is $x^+:=x\vee0$, the negative part of $x$ is $x^-:=(-x)\vee0$ and absolute value of $x$ is $\lvert x\rvert:=x\vee (-x)$. A vector lattice is called order complete if every nonempty bounded above subset has a supremum (or, equivalently, whenever every nonempty bounded below subset has an infimum). A vector lattice is order complete iff $0\leq x_n\uparrow\leq x$ implies the existence of $\sup{x_n}$. For a positive element $a$ in a vector lattice $E$, the set $[-a,a]=\{x\in E:-a\leq x\leq a \}$ is an order interval. Also, if any subset in $E$ is included in an order interval then it is called order bounded set. An order bounded operator between vector lattices $E$ and $F$ sends order bounded subsets to order bounded subsets, abbreviated as $L_b(E;F)$; for more detail information on these notions see \cite{AB,ABPO,AAydn2,Aybdd,Riez,Tr,Za}. 

Recall that a subset $A$ of a vector lattice $E$ is called solid if, for each $x\in A$ and $y\in E$,  $|y|\leq|x|$ implies $y\in A$. A solid vector subspace of a vector lattice is referred to as an ideal. An order closed ideal is called a band. Let $E$ be vector lattice and $\tau$ be a linear topology on it. Then $(E,\tau)$ is called a {\em locally solid vector lattice} (or, {\em locally solid Riesz space}) if $\tau$ has a base which consists of solid sets, for more details on these notions see \cite{AB,ABPO,Za}. A vector lattice $E$ is called \textit{Archimedean} whenever $\frac{1}{n}x\downarrow 0$ holds in $E$ for each $x\in E_+$. In this article, unless otherwise, all vector lattices are assumed to be real and Archimedean.

Consider a subset $K$ of the set $\mathbb{N}$ of all natural numbers. Let's define a new set $K_n=\{k\in K:k\leq n\}$. Then we denote $\lvert K_n\rvert$ for the cardinality of the set $K_n$. If the limit of $\mu(K)=\lim\limits_{n\to\infty}\lvert K_n\rvert/n$ exists then $\mu(K)$ is called the asymptotic density of the set $K$. Let $X$ be a topological space and $(x_n)$ be a sequence in $X$. Then $(x_n)$ is said to be statistically convergent to $x\in X$ whenever, for each neighborhood $U$ of $x$, we have $\mu\big(\{n\in\mathbb{N}:x_n\notin U\}\big)=0$; see for example \cite{AP,MK}. Similarly, a sequence $(x_n)$ in a locally solid Riesz space $(E,\tau)$ is said to be statistically $\tau$-convergent to $x\in E$ if it is provided that, for every $\tau$-neighborhood $U$ of zero, $\lim\limits_{n\to\infty}\frac{1}{n}\big\lvert\{k\leq n:(x_k-x)\notin U\}\big\rvert=0$ holds. Let $(x_n)$ be a sequence in a locally solid Riesz space $(E,\tau)$. If there exists some scalar $\lambda>0$ such that $\mu\big(\{n\in\mathbb{N}:\lambda x_n\notin U\}\big)=0$ holds for every  $\tau$-neighborhood $U$ of zero then we say that $(x_n)$ is  statistically $\tau$-bounded; see \cite{AP}. 


\section{Boundedness}
In this section, we give the notion statistical bounded set and introduce $st$-$bo$-, $st$-$bb$- and $st$-$bs$-bounded operators. The following notion could be known, but since we do not have a reference for which we give that without reference.
\begin{defn}
Let $(X,\tau)$ be a topological vector space. A subset $B\subseteq X$ is called statistical bounded (or shortly, $st$-bounded) set in $X$ if, for every zero neighborhood $U$, there is a scalar $\lambda>0$ such that
$$
\lim\limits_{n\to\infty}\frac{1}{n}\big\lvert\{b\in B:\lambda b\notin U\}\big\rvert=0.
$$
\end{defn}

It is clear that subset of an $st$-bounded set is $st$-bounded. On the other hand, a subset $B$ in a topological vector space $(E,\tau)$ is called topological bounded (or shortly, $\tau$-bounded) if, for every zero neighborhood $U$ in $E$, there is a positive scalar $\lambda$ such that $B\subseteq \lambda U$. So, we have the following useful fact.
\begin{lem}\label{top bounded is st bounded}
Every $\tau$-bounded set in topological vector spaces is $st$-bounded.
\end{lem}

Let $X$ and $Y$ be topological vector spaces. An operator $T:X\to Y$ is said to be $bb$-bounded if it maps every bounded set into a bounded set; see \cite{Tr}. Motivated by this definition and by the $ob$-bounded operator in \cite{Aybdd} and the $nb$- and $bb$-bounded operators in \cite{HMZ}, we can give the following notions.

\begin{defn}
Let $E$ be a topological vector space, $F$ be an ordered topological vector space and $T:E\to F$ be an operator. Then $T$ is said to be 
\begin{enumerate}
\item[(1)] $st$-$bo$-bounded if it maps every statistically bounded set into order bounded set,
\item[(2)] $st$-$bb$-bounded if it maps each statistically bounded set into a topological bounded set,
\item[(3)] $st$-$bs$-bounded if it maps statistically bounded sets into statistically bounded sets.
\end{enumerate}
\end{defn}

Consider the following theorems which are two classical results about order bounded subsets in a locally solid vector lattice; see \cite[Theorem 2.19.(i)]{AB} and \cite[Theorem 2.2.]{L}, respectively.
\begin{thm}\label{order bdd is top. bdd }
Let $(E,\tau)$ be a locally solid vector lattice. Then every order bounded subset in $E$ is $\tau$-bounded.
\end{thm}

\begin{thm}\label{top. bdd is order bdd}
Let $(E,\tau)$ be an ordered topological vector space that has order bounded $\tau$-neighborhood of zero. Then every $\tau$-bounded subset is order bounded.
\end{thm}

We continue with the following several basic results that follow directly from their basic definitions and Lemma \ref{top bounded is st bounded}, Theorem \ref{order bdd is top. bdd } and Theorem \ref{top. bdd is order bdd}, so its proof is omitted.
\begin{rem} 
Let $E$ be a topological vector space and $F$ be an ordered topological vector space. Then we have that
\begin{enumerate}
\item[(i)] every ordered bounded set in a locally solid vector lattice is $st$-bounded,

\item[(ii)] the $st$-$bb$-boundedness implies $st$-$bs$-boundedness,

\item[(iii)] the $st$-$bo$-boundedness implies $st$-$bb$-boundedness if $F$ is a locally solid vector lattice,

\item[(iv)] the $st$-$bo$-boundedness implies $st$-$bs$-boundedness if $F$ is a locally solid vector lattice,

\item[(v)] the $st$-$bb$-boundedness implies $st$-$bo$-boundedness whenever $F$ has order bounded $\tau$-neighborhood of zero,

\item[(vi)] every ordered bounded operator is $st$-$bo$-bounded if $E$ is a locally solid vector lattice,

\item[(vii)] ordered boundedness implies $st$-$bo$-boundedness implies $st$-$bb$-boundedness implies $st$-$bs$-boundedness  if $E=F$ is a locally solid vector lattice, so we have $L_b(E)\subseteq B_{st-bo}(E)\subseteq B_{st-bb}(E)\subseteq B_{st-bs}(E)$,

\item[(vii)] if $E=F$ is a locally solid vector lattice with an order bounded $\tau$-neighborhood of zero then the $st$-$bo$-boundedness and the $st$-$bb$ are coinciding, so we have $L_b(E)\subseteq B_{st-bo}(E)=B_{st-bb}(E)\subseteq B_{st-bs}(E)$.
\end{enumerate}
\end{rem} 

We denote $B_{st-bo}(E,F)$ the class of all $st$-$bo$-bounded operators from a topological vector space $E$ to an ordered topological vector space $F$. Also, we denote $B_{st-bb}(E,F)$ and $B_{st-bs}(E,F)$ the class of all $st$-$bb$-bounded and the class of all $st$-$bb$-bounded operators from $E$ to $F$, respectively. Now, we can allocate the sets $B_{st-bo}(E,F)$, $B_{st-bb}(E,F)$ and $B_{st-bs}(E,F)$ to the topology of uniform convergence on $st$-bounded sets. That is, a net $(T_\alpha)$ in $B_{st-bo}(E,F)$ (respectively, in $B_{st-bb}(E,F)$ or $B_{st-bs}(E,F)$) converges to zero on an $st$-bounded set $B$ with this topology if, for each $u\in F_+$(respectively, for every zero neighborhood $U$ or for each $st$-bounded zero neighborhood $V$ of $F$), there is an index $\alpha_0$ such that $T(B)\subseteq [-u,u]$ (respectively, $T(B)\subseteq U$ or $T(B)\subseteq V$) for all $\alpha\geq \alpha_0$. We should note that these classes of operators are not equal, in general. To see that let's consider the following example \cite[Example 2.3.]{L}.
\begin{exam}\
Let $E=F$ be the ordered topological vector space $\mathbb{R}^2$ with the lexicographic ordering and the usual topology. Then the identity operator $I$ on $E$ is $st$-$bo$-bounded, but it fails to be $st$-$bb$-bounded. Indeed, consider two elements $x=(-1,0)$ and $y=(1,0)$ in $E$. Then the order interval $[x,y]$ is clearly an order bounded in $E$. It can be seen that this interval includes uncountably many infinite vertical rays. Hence, the order interval $[x,y]$ is not topological bounded set in $E$ because $E$ is the usual topology and a topological bounded set cannot contain any infinite vertical ray. Thus, $st$-$bo$-boundedness does not imply $st$-$bb$-boundedness in general.
\end{exam}

Also, we can consider the following example.
\begin{exam}
Let $E=F$ be $c_0$, the space of all null sequences with the usual order and norm topology. Then the identity operator $I$ on $E$ is $st$-$bb$-bounded, but it fails to be $st$-$bo$-bounded. Indeed, consider the unit ball $B(0,1)$ centered at zero with radius one of $c_0$. Let's take the sequence $(x_n)$ defined by $x_n=(1,1,1,\dots,1,0,0,0\dots)$. Then it can be seen that $(x_n)$ is $st$-bounded because topological bounded. But it is not order bounded in $c_0$. Thus, $st$-$bb$-boundedness does not imply $st$-$bo$-boundedness in general.
\end{exam}

\section{Main Results}
We give some results about $st$-$bo$-, $st$-$bb$- and $st$-$bs$-bounded operators. In ordered vector spaces, the summation of finite order bounded sets is order bounded, and the summation of finite topological bounded sets is bounded in topological vector spaces. Thus, the summation of $st$-$bo$-bounded operators is also $st$-$bo$-bounded and the summation of $st$-$bb$-bounded operators is also $st$-$bb$-bounded. For the $st$-$bs$-bounded operators, we have the following proposition.

\begin{prop}
The summation of $st$-$bs$-bounded operators $S$ and $T$ from a topological vector space $E$ to a locally solid vector space $F$ is $st$-$bs$-bounded.
\end{prop}
 
\begin{proof}
It is enough to show that the summation of finite statistical bounded sets in locally solid vector lattice is $st$-bounded. Consider two $st$-bounded sets $B_1$ and $B_2$ in a locally solid vector lattice  $F$. For any zero neighborhood $U$, there exists another zero neighborhood $W$ such that $W+W\subseteq U$. Hence, we have some positive scalars $\lambda_{w_1},\lambda_{w_2}>0$ such that $\mu(K_{w_1})=\mu(K_{w_2})=0$ for sets $K_{w_1}=\{x\in B_1:\lambda_{w_1} x\notin W\}$ and $K_{w_2}=\{y\in B_2:\lambda_{w_2} y\notin W\}$. Let choose a positive scalar $\lambda=\min\{\lambda_{w_1},\lambda_{w_2},1\}$. Then we have $\lvert\lambda x\rvert\leq\lvert\lambda_{w_1} x\rvert$ and $\lvert\lambda y\rvert\leq\lvert\lambda_{w_2} y\rvert$ for $x\in B_1$ and $y\in B_2$, respectively. Since $W$ is solid and absorbing, we have $\lambda x, \lambda y\in W$. Then we have 
$$
\lambda x+\lambda y\in W+W\subseteq U.
$$
Thus we get $\mu\big(\{x+y:x\in B_1,y\in B_2 \ and \ \lambda(x+y)\notin U\}=0$. Consequently, $B_1+B_2$ is a statistically bounded set.
\end{proof}

Firstly, we show the continuity of addition and scalar multiplication with the uniform convergence topology.
\begin{thm}
The operations of addition and scalar multiplication are continuous in $B_{st-bo}(E,F)$, $B_{st-bb}(E,F)$ and $B_{st-bs}(E,F)$.
\end{thm}

\begin{proof}
The continuity of addition and scalar multiplication in $B_{st-bs}(E,F)$ can be shown in the same way in the proof \cite[Theorem 2.1.]{Aybdd}. So, it is enough to show continuity for $st$-$bo$-operators.

Suppose the nets $(T_\alpha)$ and $(S_\alpha)$ are nets of $st$-$bo$-bounded operators that converges to zero uniformly on $st$-bounded sets. Fix an arbitrary $st$-bounded set $B$ in $E$. Then, for any $f\in F_+$, there are some $\lambda_1,\lambda_2>0$ such that $T_\alpha(B)\subseteq[-f,f]$ for all $\alpha\geq\alpha_1$ and $S_\alpha(B)\subseteq[-f,f]$ for each $\alpha\geq\alpha_2$. So, there is another index $\alpha_0$ such that $\alpha_0\geq\alpha_1$ and $\alpha_0\geq\alpha_1$ because the index set of nets is directed. Hence, $T_\alpha(B)\subseteq [-f,f]$ and $S_\alpha(B)\subseteq [-f,f]$ for all $\alpha\geq\alpha_0$. Then we have
$$
(T_\alpha+S_\alpha)(B)\subseteq  [-f,f]+[-f,f]=[-2f,2f]
$$
for each $\alpha\geq\alpha_0$. Now, we show convergence for the scalar multiplication. Consider the $st$-bounded set $B$ in $E$. Take a sequence of reals $(\lambda_n)$ and assume it converges to zero. Since $(T_\alpha)$ is uniform convergent to zero on $st$-bounded sets, for every positive vector $u\in F_+$, there exists an index $\alpha_0$ such that $T_\alpha(B)\subseteq [-u,u]$ for all $\alpha\geq\alpha_0$. We know that for enough large $n$, we have $\lvert\lambda_n\rvert\leq 1$, and so $\lambda_n [-u,u]\subseteq [-u,u]$. Then, for all $\alpha\geq\alpha_0$ and sufficiently large $n$, we have
$$
\lambda_nT(B)=T(\lambda_n B)\subseteq\lambda_n [-u,u]\subseteq [-u,u].
$$
Therefore, we get the desired result.
\end{proof}

We investigate the lattice operations are continuous with the topology of uniform convergence in the following result.
\begin{thm}\label{lattice operations is cont.}
Let $E$ be a topological vector space and $F$ be an order complete locally solid vector lattice. Then the lattice operations with the uniform convergence are continuous in $B_{st-bo}(E,F)$, $B_{st-bb}(E,F)$ and $B_{st-bs}(E,F)$ on statistically bounded sets.
\end{thm}

\begin{proof}
We show the continuity of the lattice operations in $B_{st-bs}(E,F)$. The other cases can be shown similarly. Suppose the nets $(S_\alpha)$ and $(T_\alpha)$ of $st$-$bs$-bounded  operators converge to the linear operators $S$ and $T$ uniformly on $st$-bounded sets, respectively. By applying \cite[Theorem 1.8]{ABPO}, we have
$$
\big(S\vee T\big)(x)=\sup\{Sy+Tz:y+z=x \ and \ y,z\in E_+\}
$$
for every $x\in E_+$. 

Let's fix an $st$-bounded set $B$ in $E$. So, consider another set $A=\{y\in E_+:\exists z\in E_+, x=y+z \ \text{for some}\ x\in B_+\}$. Then $A$ is also $st$-bounded. Indeed, for any zero neighborhood $U$, there exists a positive scalar $\lambda$ such that $\mu(K_B)=0$ for $K_B=\{x\in B: \lambda x\notin U\}$ since $B$ is $st$-bounded. Consider the set $K_A=\{y\in A:\lambda y\notin U\}$. For a vector $y\in A$, we have $x\in B_+$ and $z\in E_+$ such that $y+z=x$, and so $\lambda y\leq\lambda x$. Then, by using the solidness of $U$, we get $y\in K_A$ whenever $x\in K_B$. Thus, one can see that the carnality of $K_A$ can not be bigger than two times the carnality of $K_B$ since we have $z\in K_A$ if $y\in K_A$. As a result, we get $\mu(K_A)=0$ because of $\mu(K_B)=0$. Therefore, $A$ is statistically bounded set in $E$. Thus, we have that $S_\alpha\to S$ and $T_\alpha\to T$ converge uniformly on $A$. Take a fixed $x\in B_+$ and some elements $y,z\in E_+$ such that $x=y+z$. So, by the formula $\sup(M)-\sup(N)\leq \sup(M-N)$ for any set $M,N$, we have the following inequality
\begin{eqnarray*}
\big(S_\alpha\vee T_\alpha\big)(x)-\big(S\vee T\big)(x)&=&\sup\{S_\alpha y+T_\alpha z:y+z=x \ and \ y,z\in E_+\}\\&& -\sup\{Sy+Tz:y+z=x \ and \ y,z\in E_+\}\\&\leq&\sup\{(S_\alpha-S)y+(T_\alpha-T)z:y+z=x \ and \ y,z\in E_+\}.
\end{eqnarray*}
Now, consider any $st$-bounded zero neighborhood $V$ in $F$. Then there is another zero neighborhood $W$ such that $W+W\subseteq V$. One can see that $W$ is also $st$-bounded. Hence, since $S_\alpha\to S$ and $T_\alpha\to T$ converge uniformly on $A$ and indexed set is directed, we have some index $\alpha_0$ such that $(S_\alpha-S)(A)\subseteq W$ and $(T_\alpha-T)(A)\subseteq W$ for all $\alpha\geq\alpha_0$. Thus, we have 
$$
\big(S_\alpha\wedge T_\alpha\big)(x)-\big(S\wedge T\big)(x)\leq(S_\alpha-S)(x)+(T_\alpha-T)(x)\in W+W\subseteq V.
$$
for all $\alpha\geq\alpha_0$. Therefore, $S_\alpha\wedge T_\alpha-S\wedge T$ is $st$-$bs$-bounded.
\end{proof}
 
It is clear that the product of $st$-$bs$-bounded operators is continuous with the topology of uniform convergence on $st$-bounded sets because every subset of an $st$-bounded set is $st$-bounded. For the other cases, we give the next theorems.
\begin{thm}
Let $E$ be a topological vector space, $F$ be a locally solid vector lattice and $T:E\to F$ be an operator. Then the product of $st$-$bo$-bounded operators is continuous in $B_{st-bo}(E,F)$ with the topology of uniform convergence on $st$-bounded sets.
\end{thm}

\begin{proof}
Suppose the nets of $st$-$bo$-bounded operators $(T_\alpha)$ and $(S_\alpha)$ converge to zero uniformly on $st$-bounded sets. Let's take an $st$-bounded set $B$. Then, for fixed positive vector $u\in E_+$, there is an index $\alpha_1$ such that $T_\alpha(B)\subseteq [-u,u]$ for all $\alpha\geq \alpha_1$. Since every order bounded set in $E$ is topological bounded, we get that $[-u,u]$ is $\tau$-bounded, and so $st$-bounded; see Theorem \ref{order bdd is top. bdd }. Thus, $T_\alpha(B)$ is also $st$-bounded for each $\alpha\geq \alpha_1$. Then $S_\alpha(T_\alpha(B))$ is order bounded for each $\alpha\geq \alpha_1$. So, for a given positive $u$, we have another index $\alpha_2\geq \alpha_1$ such that
$$
S_\alpha\big(T_\alpha(B)\big)\subseteq [-u,u]
$$
for all $\alpha\geq \alpha_2$. Therefore, we get the desired result.
\end{proof}

\begin{prop}\label{t is o}
Let $E$ be a topological vector space, $F$ be an ordered topological vector space. If a net $(T_\alpha)$ of $st$-$bs$-bounded operators order converges uniformly on some $st$-bounded sets to na operator $T$ then $T$ is $st$-$bs$-bounded.
\end{prop}

\begin{proof}
Suppose $B$ is $st$-bounded set in $E$. Then, for each $st$-bounded zero neighborhood $U$ in $F$, there exists an index $\alpha_0$ such that $(T-T_\alpha)(B)\subseteq U$ for each $\alpha\geq \alpha_0$. So, we get
$$
T(B)\subseteq T_{\alpha_0}(B)+U.
$$
Since $T_{\alpha_0}$ is $st$-$bs$-bounded, we have $T_{\alpha_0}(B)$ is also $st$-bounded. Thus, $T_{\alpha_0}(B)+U$ is also $st$-bounded because sum of $st$-bounded sets is $st$-bounded. As a result, $T(B)$ is $st$-bounded, and so we get the desired result.
\end{proof}

\begin{prop}\label{t is order}
Let $E$ be a topological vector space, $F$ be an ordered topological vector space. If a net $(T_\alpha)$ of $st$-$bo$-bounded operators order converges uniformly on some $st$-bounded sets to an operator $T$ then $T$ is $st$-$bo$-bounded.
\end{prop}

\begin{proof}
Let $B$ be $st$-bounded set in $E$. Then, for each positive vector $u$ in $F$, there exists an index $\alpha_0$ such that $(T-T_\alpha)(B)\subseteq [-u,u]$ for each $\alpha\geq \alpha_0$. So, we get
$$
T(B)\subseteq T_{\alpha_0}(B)+[-u,u].
$$
Since $T_{\alpha_0}$ is $st$-$bo$-bounded, we have $T_{\alpha_0}(B)$ is also order bounded set in $F$. Thus, $T_{\alpha_0}(B)+[-u,u]$ is also order bounded because the sum of order bounded sets is order bounded. As a result, $T(B)$ is order bounded set in $F$.
\end{proof}

\begin{thm}
Let $E$ be a topological vector space, $F$ be an ordered topological complete vector space. If every $st$-bounded set in $E$ is absorbing then $B_{st-bs}(E,F)$ is complete with the topology of uniform convergence on $st$-bounded sets.
\end{thm}

\begin{proof}
Let $(T_\alpha)$ be a Cauchy sequence in $B_{st-bs}(E,F)$ with respect to uniform convergence on $st$-bounded sets. consider an arbitrary $st$-bounded set $B$ in $E$. Hence, for an $st$-bounded zero neighborhood $U$ in $F$, there exists an $\alpha_0$ such that $(T_\alpha-T_\beta)(B)\subseteq U$ for each $\alpha,\beta\geq\alpha_0$. Next, take arbitrary vector $x$ in $E$. Then since $B$ is absorbing, there is some positive scalar $\lambda$ such that $\lambda x\in B$. 
So, we get $(T_\alpha(x)-T_\beta)(x))\subset\frac{1}{\lambda}U$ for every $\alpha,\beta\geq\alpha_0$. So, $(T_\alpha(x))$ is a Cauchy net in $F$. Then, by choosing an operator $T$ from $E$ to $F$ as $T(x)$ is the limit of $(T_\alpha(x))$ and by applying Proposition \ref{t is o}, $T$ is $st$-$bs$-bounded.
\end{proof}

\begin{thm}
Let $E$ be a topological vector space, $F$ be an ordered topological complete vector space. If every $st$-bounded set in $E$ is absorbing then $B_{st-bo}(E,F)$ is complete with the topology of uniform convergence on $st$-bounded sets.
\end{thm}

\begin{proof}
Suppose $(T_\alpha)$ is a Cauchy sequence in $B_{st-bo}(E,F)$ with respect to uniform convergence on $st$-bounded sets. Fix an $st$-bounded set $B$ in $E$. Thus, for every positive vector $u\in F_+$, there is an index $\alpha_0$ such that $(T_\alpha-T_\beta)(B)\subseteq [-u,u]$ for each $\alpha,\beta\geq\alpha_0$. 

Now, let's take an element $x\in E$. Then we have some positive scalar $\lambda$ such that $\lambda x\in B$ because $B$ is absorbing. So, we get $(T_\alpha(x)-T_\beta)(x))\subset[\frac{-u}{\lambda},\frac{u}{\lambda}]$ for every $\alpha,\beta\geq\alpha_0$. Thus, we have that $(T_\alpha(x))$ is a Cauchy net in $F$. So, one can choose an operator $T$ from $E$ to $F$ as $T(x)$ is the limit of $(T_\alpha(x))$. Then, by applying Proposition \ref{t is order}, we get $T\in B_{st-bo}(E,F)$.
\end{proof}

\end{document}